\documentclass[11pt]{article}

\usepackage{amsmath,amssymb,amsthm,fullpage}
\usepackage{hyperref}
\usepackage{graphicx}
\usepackage{booktabs}
\usepackage{float, caption}
\usepackage{color}

\newtheorem{theorem}{Theorem}
\newtheorem{lemma}{Lemma}
\newtheorem{remark}{Remark}
\newtheorem{corollary}{Corollary}

\newtheorem{definition}{Definition}

\title{\textbf{Saddle Networks}:\\Structure-Preserving Architectures for Convex-Concave Functions}
\author{
Xavier Warin\\
EDF Lab Paris-Saclay and FiMe, Laboratoire de Finance des March\'es de l'Energie\\
91120 Palaiseau, France\\
\texttt{xavier.warin@edf.fr}
}

\begin{document}
\maketitle

\begin{abstract}
Saddle-point models arise throughout optimization, optimal transport, robust learning, and control.
In many applications, the relevant function \(f(x,y)\) is convex in \(x\) and concave in \(y\), and preserving this geometry is essential for obtaining tractable min--max formulations and reliable certificates.
We introduce a structured separable decomposition that preserves the convex-concave geometry and prove a complete one-dimensional approximation theorem under a mixed Monge-type convexity condition.
We then describe practical \emph{saddle network} architectures that preserve convexity in \(x\) and concavity in \(y\) by construction. The proposed architectures require only convexity-preserving neural networks, together with simple output transformations enforcing sign and concavity constraints.
Finally, we report numerical benchmarks in dimension \(1\) and  \(5\), showing that the proposed saddle networks achieve high accuracy on smooth, nonsmooth, and high-rank convex--concave test functions.
\end{abstract}

\section{Introduction}

Many decision problems are naturally organized around a function
\[
f(x,y),
\]
which is convex in a decision variable \(x\) and concave in an adversarial,
strategic, or dual variable \(y\). Such functions define saddle-point problems
of the form
\[
\min_x \max_y f(x,y),
\]
and their geometry is essential for well-posedness and algorithmic stability. Convex-concave saddle formulations
appear in robust optimization, minimax learning, game theory, machine learning,
and finance, and have recently motivated dedicated modeling languages such as
disciplined saddle programming \cite{SchieleLuxenbergBoydDSP,JuditskyNemirovski2022}.

The most familiar examples arise from Lagrangian duality, where \(y\) is a
multiplier and the interaction between \(x\) and \(y\) is often affine. However,
many modern applications require richer convex-concave functions that cannot be
reduced to a simple Lagrangian coupling. In zero-sum games and strategic
multi-agent models, \(f(x,y)\) may represent a learned payoff surface, with
\(x\) and \(y\) corresponding to the actions of competing players. In robust and
adversarial learning, \(y\) may represent a perturbation, a stress scenario, or
an adversarial distribution, leading to minimax objectives in which the
convex-concave structure is exploited by first-order saddle-point algorithms
\cite{YangZhangKiyavashHe2020,LinJinJordan2020}. In distributionally robust
learning and related prediction tasks, the maximization variable can encode
uncertainty over distributions rather than a finite-dimensional Lagrange
multiplier, so that the learned coupling between \(x\) and \(y\) may be
nonlinear and problem-specific.

Convex-concave functions also arise in online learning and resource allocation.
The online saddle-point problem generalizes online convex optimization by asking
both players to compete against the saddle value of the accumulated payoff
function. This framework connects to online convex optimization with knapsack
constraints and has applications in dynamic pricing, auctions, and crowdsourcing
\cite{RiveraWangXu2020}. In fairness-aware machine learning, minimax group
fairness can be formulated by maximizing over group weights while minimizing over
model parameters; convex-concave saddle formulations provide a structured way
to balance global performance against worst-group performance
\cite{BotCsetnekSedlmayer2022}. These examples require representing a payoff or
risk function whose interaction term is not necessarily bilinear and whose shape
constraints should be preserved during learning.

Another important source of convex-concave structure is control and differential
games. In two-player zero-sum stochastic control, Hamilton--Jacobi--Isaacs
equations involve min--max or max--min nonlinearities, and the corresponding
Hamiltonians encode the interaction between a controller and an adversary
\cite{KaweckiSmears2022,EvansSouganidis1984}. Learning or approximating such
Hamiltonians with unconstrained neural networks may destroy the saddle geometry,
whereas a convex-concave architecture can preserve the variational structure
needed for stable downstream optimization. This motivates neural surrogates that
are not merely accurate in mean-square error, but also preserve the convexity in
the minimizing variables and concavity in the maximizing variables.

Classical convex analysis provides several structure-preserving tools, including
Moreau-type smoothing, partial conjugacy, and self-dual constructions
\cite{Goebel2008}. These methods are analytically powerful, but they do not
always yield parametric surrogates that are convenient for learning pipelines or
large-scale numerical optimization. Neural architectures offer an alternative:
shape constraints can be imposed directly by construction. Input Convex Neural
Networks (ICNNs) enforce convexity by constraining hidden-to-hidden weights to be
nonnegative and using convex nondecreasing activations \cite{AmosICNNPMLR}.
Such models have been used to combine expressive learning with convex planning,
for example in convex approaches to optimal control and model predictive control
\cite{ChenShiZhang2019}. More recent architectures, such as COMONet, aim to
integrate several shape constraints, including monotonicity, convexity,
concavity, and their combinations, in a unified network design
\cite{COMONet2025}.

Our goal is to develop neural architectures for representing and approximating
convex-concave functions while preserving their saddle geometry by construction.
We focus on the idea that one can build saddle networks from
convexity-preserving primitives. Concavity is obtained by sign changes, while
positivity and negativity constraints are enforced by simple output
transformations or range-control mechanisms. This viewpoint is compatible with
ICNNs \cite{AmosICNNPMLR} and input-convex Kolmogorov--Arnold network variants
such as ICKAN \cite{ICKANArxiv}. 

The contributions of this article are as follows. First, we formulate a
structure-preserving saddle architecture based on products of signed convex and
concave factors together with additive convex and concave marginals. Second, we
prove a complete universal approximation result in dimension one under a mixed convexity condition, relying on a
discrete saddle decomposition for piecewise affine functions.  Finally, we describe
practical neural realizations based on ICNN and ICKAN primitives and
evaluate them on smooth, nonsmooth, and higher-dimensional convex-concave
benchmarks comparing the result obtained with the results obtained by the COMONet.

\section{Theory: Structured Saddle Decomposition}
Throughout the paper, the superscript \(cv\) denotes convexity and \(cc\) denotes concavity. 
The superscripts \(+\) and \(-\) indicate nonnegativity and nonpositivity, respectively. 
For instance, \(g^{cv,+}\) is nonnegative and convex, while \(g^{cc,-}\) is nonpositive and concave.

Let $X\subset\mathbb{R}^d$ and $Y\subset\mathbb{R}^d$ be nonempty compact convex sets. Let $f:X\times Y\to\mathbb{R}$ be continuous and assume:
for every fixed $y\in Y$, the map $x\mapsto f(x,y)$ is convex, and for every fixed $x\in X$, the map $y\mapsto f(x,y)$ is concave.
Our goal is to approximate this function $f$  while preserving its convex-concave structure.
In dimension 1 we can prove the following result:

\begin{theorem}[One-dimensional saddle approximation under the Monge condition]
\label{thm:main}
Let \(X,Y\subset\mathbb R\) be compact intervals. Let
\(f\in C(X\times Y)\) be such that:

\begin{enumerate}
\item for every fixed \(y\in Y\), the map \(x\mapsto f(x,y)\) is convex;

\item for every fixed \(x\in X\), the map \(y\mapsto f(x,y)\) is concave;

\item the mixed Monge condition holds:
\[
\partial_{xx}\bigl(-\partial_{yy}f\bigr)\ge 0
\]
in the sense of distributions on \(X\times Y\).
\end{enumerate}

Then, for every \(\varepsilon>0\), there exist an integer \(N\ge 1\)  such that
\begin{align*}
\sup_{(x,y)\in X\times Y}
\Bigg|
f(x,y)
&- \left(
\sum_{i=1}^N e_i^{cv,+}(x)a_i^{cc,+}(y)
+
G(y) \right)
\Bigg|
<\varepsilon .
\end{align*}
    where
    \[
\begin{array}{lll}
e_i^{cv,+}:X\to\mathbb R^+ \quad \text{convex},&
 a_i^{cc,+}:Y\to\mathbb R^+ \quad \text{concave,}& 
 G:Y\to\mathbb R \quad \text{concave.}
 \end{array}
\]
\end{theorem}
\begin{proof}
    The proof is deferred to the appendix.
\end{proof}
\begin{remark}
The additional mixed convexity condition is necessary in general.
Without it, convex-concave functions need not admit a finite
structured separable decomposition of the above type.
\end{remark}
\begin{remark}
    The proof   relies crucially on the one dimensional structure of each part for the function and does not extend to $d>1$.
\end{remark}
\begin{corollary}
\label{corol:1}
Let as $f$ in theorem \ref{thm:main} but  following the mixed Monge condition
\[
\partial_{xx}\bigl(-\partial_{yy}f\bigr)\le 0
\]
in the sense of distributions on \(X\times Y\),
then  for every \(\varepsilon>0\), there exist an integer \(N\ge 1\)  such that
\begin{align*}
\sup_{(x,y)\in X\times Y}
\Bigg|
f(x,y)
&- \left(
\sum_{i=1}^N e_i^{cv,-}(x)a_i^{cv,+}(y)
+
H(x) \right)
\Bigg|
<\varepsilon .
\end{align*}
    where
    \[
\begin{array}{lll}
e_i^{cv,-}:X\to\mathbb R^-,&
 a_i^{cv,+}:Y\to\mathbb R_+ ,&
 H:X\to\mathbb R  \quad \text{are all convex.}
 \end{array}
\]
\end{corollary}
\begin{proof}
Define \(F(y,x):=-f(x,y)\). Then \(F\) is convex in its first variable \(y\) and concave in its second variable \(x\). Moreover, the assumption
\[
\partial_{xx}\bigl(-\partial_{yy}f\bigr)\le 0
\]
is equivalent to the Monge condition required by Theorem~\ref{thm:main} for \(F\), with the roles of \(x\) and \(y\) interchanged.

Applying Theorem~\ref{thm:main} to \(F\), we obtain
\[
\sup_{(x,y)\in X\times Y}
\left|
-f(x,y)
-
\left(
\sum_{i=1}^N \widehat e_i^{cv,+}(y)\widehat a_i^{cc,+}(x)
+
\widehat G(x)
\right)
\right|
<\varepsilon .
\]
Multiplying by \(-1\), the result follows by setting
\[
e_i^{cv,-}(x):=-\widehat a_i^{cc,+}(x),\qquad
a_i^{cv,+}(y):=\widehat e_i^{cv,+}(y),
\qquad
H(x):=-\widehat G(x).
\]
Indeed, \(-\widehat a_i^{cc,+}\) is convex and nonpositive, while \(H\) is convex.
\end{proof}

 The previous theorem and corollary motivate the following general saddle class :
\begin{definition}
\label{def:saddle}
Let \(X,Y\subset\mathbb R^d\).  
The saddle class of order \(N\) on \(X\times Y\) is the set of functions of the form
\[
f(x,y)
=
\sum_{i=1}^N e_i^{cv,+}(x)a_i^{cc,+}(y)
+
\sum_{i=1}^N e_i^{cv,-}(x)a_i^{cv,+}(y)
+
H(x)+G(y),
\]
where
\[
e_i^{cv,+}:X\to\mathbb R_+
\quad\text{is convex},
\qquad
e_i^{cv,-}:X\to\mathbb R_-
\quad\text{is convex},
\]
\[
a_i^{cc,+}:Y\to\mathbb R_+
\quad\text{is concave},
\qquad
a_i^{cv,+}:Y\to\mathbb R_+
\quad\text{is convex},
\]
and where
\[
H:X\to\mathbb R
\quad\text{is convex},
\qquad
G:Y\to\mathbb R
\quad\text{is concave}.
\]
\end{definition}
Every function in the saddle class is convex in \(x\) and concave in \(y\). Indeed, for fixed \(y\), all coefficients multiplying the convex functions of \(x\) are nonnegative. For fixed \(x\), the first sum is a nonnegative multiple of concave functions of \(y\), while the second sum is a nonpositive multiple of convex functions of \(y\), and is therefore concave in \(y\).

The previous saddle-class definition can be extended giving  a second class of functions:
\begin{definition}
\label{def:saddleGen}
Let \(X,Y\subset\mathbb R^d\).  
The bilinear saddle class of order \(N\) on \(X\times Y\) is the set of functions of the form
\[
f(x,y)=g(x,y)+x^\top B y,
\]
where \(g\) belongs to the saddle class of order \(N\) and \(B\in\mathbb R^{d\times d}\).
\end{definition}
\begin{remark}
On compact sets, the bilinear term \(x^\top B y\) can itself be represented within a saddle class of sufficiently large order, up to affine marginal terms. Indeed, for each pair \((k,\ell)\), choose constants \(\alpha_\ell,\beta_k\) such that
\[
y_\ell+\alpha_\ell\ge 0,\qquad x_k+\beta_k\ge 0
\]
on \(Y\) and \(X\), respectively. If \(B_{k\ell}>0\), then
\[
B_{k\ell}x_k y_\ell
=
B_{k\ell}(x_k+\beta_k)(y_\ell+\alpha_\ell)
-
B_{k\ell}\alpha_\ell x_k
-
B_{k\ell}\beta_k y_\ell
-
B_{k\ell}\alpha_\ell\beta_k.
\]
The product term is of type \(e^{cv,+}a^{cc,+}\), while the remaining terms are affine marginals. If \(B_{k\ell}<0\), an analogous decomposition uses a nonpositive convex factor in \(x\) and a nonnegative convex factor in \(y\). Thus the explicit bilinear term does not enlarge the theoretical class on compact domains, but it provides a useful low-rank inductive bias and improves numerical efficiency.
\end{remark}

\begin{remark}
    The additional bilinear term in the bilinear saddle-class is the term use in \cite{COMONet2025} to model the $x-y$ interaction.
\end{remark}

From the theorem in dimension $d=1$, we can derive a Universal approximation.

\begin{theorem}[Universal approximation for (bilinear) saddle functions]\label{thm:ua-saddle}
Assume the hypotheses of Theorem~\ref{thm:main}, with the mixed convexity condition of the theorem or the one in the corollary \ref{corol:1}. Suppose moreover that the chosen
\emph{convexity-preserving network class} $\mathcal{C}$ is a universal approximator of continuous convex functions on compact subsets of $\mathbb{R}$ in the sup norm.
Then the corresponding (bilinear) saddle-network class defined by (bilinear) saddle-class functions in definition \ref{def:saddle} or \ref{def:saddleGen} where each convex or concave function is an element of $\mathcal C$  is dense in the sup norm in the class of continuous one-dimensional convex-concave functions satisfying the mixed Monge condition.
\end{theorem}

\begin{corollary}[Concrete convex network families]\label{cor:families}
Theorem~\ref{thm:ua-saddle} applies when $\mathcal{C}$ is instantiated by:
\begin{itemize}
\item ICNNs \cite{AmosICNNPMLR} (also used in convex optimal control \cite{ChenShiZhang2019});
\item GroupMax convex networks \cite{WarinGroupMaxArxiv};
\item P1-ICKAN (input-convex KAN, piecewise-linear-by-parts), which provides a universal approximation theorem in \cite{ICKANArxiv}.
\end{itemize}
For cubic-ICKAN, \cite{ICKANArxiv} reports strong numerical evidence; we treat it as an empirical alternative rather than a fully proved universal approximator.
\end{corollary}

\section{Saddle Network Architectures}\label{sec:arch}
We have shown that convex concave functions with a mixed convexity condition could be approximated in dimension 1  using the  decomposition  in definition \ref{def:saddle}  or \ref{def:saddleGen} using network preserving convexity for which a universal theorem is proved. However, we have not yet explained how  to ensure that the corresponding neural approximation remains convex-concave.
In the next sections we develop an architecture for compact set in $\mathbb R^d \times \mathbb R^d$ based on the previous results and preserving the initial convex-concave structure of function to approximate.
\subsection*{Architecture: two multi-output convex networks + post-processing}
We implement  the decompositions of Theorem~\ref{thm:ua-saddle} using two multi-output convexity-preserving networks:
\[
u(x)\in\mathbb{R}^{2N+1},\qquad v(y)\in\mathbb{R}^{2N+1}.
\]
The first $2N$ coordinates are convex outputs, split into two blocks of length $N$,
\[
u(x)=(u^{(1)}(x),u^{(2)}(x),H(x)),
\quad
v(y)=(v^{(1)}(y),v^{(2)}(y),G(y)),
\]
with the identification
\[
u^{(1)}\leftrightarrow e^{cv,+},\quad u^{(2)}\leftrightarrow e^{cv,-}, \quad
v^{(1)}\leftrightarrow -a^{cc,+},\quad v^{(2)}\leftrightarrow a^{cv,+}.
\]
All $u^{(k)}$ and $v^{(k)}$ are convex  functions by construction and we explain below how to enforce the sign constraints.
Adding the bilinear term \(x^\top B y\) yields an element of the bilinear saddle class.

\subsection*{Convex primitives: ICNN  and ICKAN}
\paragraph{ICNN (Amos--Xu--Kolter).}
ICNNs guarantee convexity by constraining hidden-to-hidden weights to be nonnegative and using convex nondecreasing activations \cite{AmosICNNPMLR}. This provides a robust baseline used in applications such as convex control \cite{ChenShiZhang2019}.
 We enforce concavity and sign constraints componentwise from a convex function $g$ with simple transforms:
\[
\text{convex }\ge 0:\ \mathrm{ReLU}(g),\qquad
\text{concave }\ge 0:\ C-g,\qquad
\text{convex }\le 0:\ g-C,
\]
where $C$ is a trainable scalar per component and violations are penalized by $\mathrm{ReLU}(g-C)$ (soft constraint).

This design emphasizes that \emph{only convexity-preserving networks are needed}: concavity is obtained by negation, and positivity/negativity is handled by ReLU or shifts.

\paragraph{ICKAN (input-convex KAN): no penalty via known range.}
ICKAN replaces parts of the network by compositions of learnable univariate  maps and is input-convex by construction \cite{ICKANArxiv}. A key practical feature is that each univariate approximation is defined on a known grid domain, and the architecture gives the \emph{image} of the grid domain. In 1D, if a layer output is guaranteed to satisfy $V\in[I_{\min},I_{\max}]$, then introducing 2 trainable variables $c$ and $d$  define
\[
V-I_{\min} + c^+\ge 0,\qquad V-I_{\max} -d^+\le 0,
\]
so that positivity/negativity constraints can be implemented \emph{without penalties} using fixed shifts based on $(I_{\min},I_{\max})$.
In higher dimension, the same idea applies coordinatewise to each scalar channel whose range is known.
This range-known mechanism (including truncation) is discussed in \cite{ICKANArxiv} and avoids the penalty terms used in the ICNN instantiations.

\section{Numerical Experiments}\label{sec:numerics}

We benchmark saddle-network architectures on a suite of convex-concave test functions.
Each test is designed to target a specific modeling difficulty (smooth coupling, steep curvature, nonsmooth corners,
 intricate $x-y$ interactions).
Throughout, each test satisfies: for every fixed $y$, $x\mapsto f(x,y)$ is convex; for every fixed $x$, $y\mapsto f(x,y)$ is concave.
We report the mean and standard deviation of the final approximation error (MSE) over $10$ independent runs.
We use PyTorch to minimize the MSE between the target function and the saddle network.
The stochastic gradient method is implemented with the Adam optimizer and a learning rate equal to \(10^{-3}\).
The batch size during training is \(2048\), and all timings are reported for an NVIDIA H100 GPU.

\subsection{1D test suite }
We test, for \(d=1\), the accuracy of the approximation for convex-concave functions with different types of features.
Unless otherwise stated, the domain is \(x,y\in[-1,1]\).
The MSE is evaluated on a grid of \(200 \times 200\) points.

\medskip
We first consider test cases with bilinear interactions :
\paragraph{Smooth saddle with bilinear coupling (baseline).}
\begin{equation}\label{eq:case1_1d}
f_1(x,y)=x^2-y^2+xy.
\end{equation}

\paragraph{Steep smooth curvature (conditioning).}
\begin{equation}\label{eq:case2_1d}
f_2(x,y)=e^x-e^y+xy.
\end{equation}

\paragraph{Smooth multi-regime transition (softplus).}
Domain: $x,y\in[-3,3]$. 
\begin{equation}\label{eq:case3_1d}
f_3(x,y)=\mathrm{softplus}(x)-\mathrm{softplus}(y)+xy,
\qquad \mathrm{softplus}(t)=\log(1+e^t).
\end{equation}

\paragraph{Nonsmooth corners (absolute value).}
\begin{equation}\label{eq:case4_1d}
f_4(x,y)=|x|-|y|+xy.
\end{equation}

\paragraph{Competing convex regimes in $x$ and kink in $y$ (max + $\ell_1$).}
\begin{equation}\label{eq:case5_1d}
f_5(x,y)=\max\{|x|,x^2\}+xy-|y|.
\end{equation}

\paragraph{Hybrid smooth/nonsmooth (Huber in $x$, kink in $y$).}
\begin{equation}\label{eq:case8_1d}
f_6(x,y)=\rho_\delta(x)+xy-|y|,
\qquad
\rho_\delta(x)=
\begin{cases}
\dfrac{x^2}{2\delta}, & |x|\le\delta,\\[3pt]
|x|-\dfrac{\delta}{2}, & |x|>\delta,
\end{cases}
\end{equation}
with $\delta=0.3$ in the implementation. 
\medskip

We then add three cases with nonlinear, non-bilinear interactions. The coefficients \(\{u_r,v_r,b_r,t_r\}\) are fixed deterministic parameters, and we set \(\varepsilon=10^{-2}\).
\paragraph{High-rank smooth interaction (softplus coupling).}
\begin{equation}\label{eq:case9_1d}
f_7(x,y)=x^2-y^2+xy
+\sum_{r=1}^{R}
\mathrm{softplus}(k\,u_r x + b_r)
\Bigl(C_r-\mathrm{softplus}(k\,v_r y + t_r)\Bigr),
\end{equation}
where $R=24$, $k=6$.
The constants $C_r$ are defined as
\[
C_r = \mathrm{softplus}\bigl(k(|v_r|+|t_r|)\bigr) + \varepsilon,
\]
ensuring that each term $C_r - \mathrm{softplus}(\cdot)$ is nonnegative.

\paragraph{Sharp nonlinear interaction (exponential coupling).}
\begin{equation}\label{eq:case10_1d}
f_8(x,y)=x^2-y^2+xy
+\sum_{r=1}^{R}
e^{k(u_r x + b_r)}
\Bigl(C_r-e^{k(v_r y + t_r)}\Bigr),
\end{equation}
where $R=16$, $k=2$, and
\[
C_r = \exp\bigl(k(|v_r|+|t_r|)\bigr) + \varepsilon.
\]

\paragraph{Polynomial high-rank interaction.}
\begin{equation}\label{eq:case11_1d}
f_9(x,y)=x^2-y^2+xy
+  \sum_{r=1}^{R}
(u_r x + b_r)^2
\Bigl(C_r-(v_r y + t_r)^2\Bigr),
\end{equation}
where $R=24$ and
\[
C_r = (|v_r|+|t_r|)^2 + \varepsilon.
\]
Each interaction term in the last three cases is constructed as the product of a convex nonnegative function of $x$ and a concave nonnegative function of $y$. The constants $C_r$ are chosen so that the concave factors remain nonnegative over the domain, ensuring that the overall function is convex in $x$ and concave in $y$.

All runs are performed using \(250000\) gradient iterations.
The different ICNNs used in our algorithm and in COMONet use \(3\) hidden layers with \(32\) neurons each.
The P1-ICKAN and Cubic-ICKAN models use \(2\) hidden layers with \(10\) neurons and \(10\) mesh intervals.
The computing time for the saddle network with ICNN is \(550\) seconds, compared with \(770\) seconds for P1-ICKAN and \(1040\) seconds for Cubic-ICKAN.
The computing time for COMONet is \(350\) seconds.

The penalty used for ICNN is equal to \(10\), and in all cases the final penalty is numerically zero, meaning that the obtained solution satisfies the required convexity and concavity constraints.
The results in Table~\ref{tab:1DICNN} show that using \(N=20\) is sufficient to obtain very accurate convergence of the  saddle network. COMONet is less accurate on this test suite, especially on nonsmooth and
high-rank nonlinear interaction cases.
When comparing the different convex-network primitives within the saddle architecture, the ICNN  appears to be the most attractive option in this experiment.
The results in Table~\ref{tab:1DICNNGen} given only with the ICNN show that adding the bilinear term may slightly improve the results.

\begin{table}[H]
    \centering
\begin{tabular}{r|cc|cc|cc|cc}
\hline
& \multicolumn{2}{c}{Saddle ICNN} &\multicolumn{2}{c}{Saddle P1-ICKAN} &\multicolumn{2}{c}{Saddle Cubic-ICKAN} &\multicolumn{2}{c}{COMONet} \\  \hline
case & mean\_mse & std\_mse &  mean\_mse & std\_mse  & mean\_mse & std\_mse & mean\_mse & std\_mse   \\
\hline
1 & 4.77e-6 & 2.53e-6 & 5.74e-06 & 8.80e-06 &2.47e-06 & 4.85e-06 &8.84e-3 & 1.05e-2\\
2 & 3.18e-6 & 1.94e-6  &  7.85e-06 & 1.15e-05 & 4.03e-06 & 6.31e-06 &6.73e-3 & 1.54e-2\\
3 & 4.56e-5 & 3.13e-5 &   1.11e-04 & 2.06e-04& 1.08e-04 & 2.51e-04& 6.04e-2 & 9.72e-2\\
4 &  1.08e-6 & 2.73e-6&  5.86e-06 & 1.07e-05 & 1.63e-06 & 1.53e-06&4.68e-2 & 9.01e-2\\
5 &  5.41e-7 & 1.03e-6 & 2.14e-05 & 6.27e-05 & 5.78e-06 & 1.04e-05&3.01e-1 & 7.04e-1 \\
6 & 1.71e-6 & 2.90e-6 & 5.06e-06 & 7.07e-06& 3.44e-06 & 4.34e-06 &2.02e-3 & 2.22e-3\\
7 &  7.37e-3 & 7.44e-3  & 6.62e-03 & 8.78e-03  & 2.56e-02 & 5.18e-02 &3.64e+0 & 1.18e-1\\
8 &  8.32e-4 & 2.73e-4 & 9.23e-03 & 1.62e-02  & 2.92e-03 & 5.27e-03  & 1.15e+0 & 2.83e-2 \\
9 &  6.56e-5 & 9.01e-5 & 2.11e-05 & 3.73e-05 & 1.34e-04 & 3.93e-04& 9.88e-2 & 1.54e-2\\
\hline
    \end{tabular}
    \caption{Convergence of the saddle networks and COMONet in dimension \(1\), using \(N=20\). High accuracy is observed for all saddle-network variants on the considered admissible test cases.
    \label{tab:1DICNN}}
\end{table}

\begin{remark}
The relatively larger errors observed in cases 7 and 8 are mainly due to the larger range of the corresponding target functions.
\end{remark}
\begin{table}[H]
    \centering
\begin{tabular}{rccccccccc}
\hline
case & 1 & 2 & 3 & 4 & 5 &6 &7 & 8 & 9 \\ \hline
mean\_mse & 1.9e-6 & 1.3e-6& 2.9e-6& 3.7e-6& 1.4e-7& 4.1e-6 & 2.66e-3 &1.6e-3&  4.1e-5\\
std\_mse & 1.3e-6 & 1.8e-6 & 1.4e-6& 7.5e-6 & 2.3e-7& 7.2e-6& 1.10e-3 &2.0e-3& 5.3e-5\\
\hline
    \end{tabular}
    \caption{Convergence of the  bilinear saddle network  in dimension \(1\), using \(N=20\) :  results are globally slightly improved.
    \label{tab:1DICNNGen}}
    
\end{table}

\subsection{Higher-dimensional test suite }
We generalize the previous tests to dimension \(d=5\).
Unless otherwise stated, \(x,y\in[-1,1]^5\).
The MSE is estimated by Monte Carlo using \(500000\) samples. Unless otherwise specified, all matrices used in the test cases are fixed across runs.

\medskip
We first define test cases with bilinear interactions::
\begin{itemize}
    \item \begin{equation}\label{eq:case1_nd}
f_1(x,y)=\|x\|_2^2-\|y\|_2^2 + x^\top B y,
\end{equation}

\item \begin{equation}\label{eq:case2_nd}
f_2(x,y)=\sum_{i=1}^d e^{x_i}-\sum_{i=1}^d e^{y_i}+x^\top B y.
\end{equation}
\item 
\begin{equation}\label{eq:case3_nd}
f_3(x,y)=\log\!\Big(\sum_{k=1}^{m} e^{(Ax)_k}\Big)
-\log\!\Big(\sum_{k=1}^{m} e^{(Cy)_k}\Big)
+x^\top B y,
\end{equation}
where $A,C\in\mathbb{R}^{m\times d}$ (with $m=8$ in the code) are fixed per run.
\item 
\begin{equation}\label{eq:case7_nd}
f_4(x,y)=\sum_{i=1}^{d}\rho_\delta(x_i) + x^\top B y - \|y\|_1,
\end{equation}
with $\delta=0.3$.
\item  for $ x, y \in [0.1,0.9]^5$,
\begin{equation}\label{eq:case8_nd}
f_5(x,y)= -\sum_{i=1}^{d}\log(x_i) + x^\top B y + \sum_{i=1}^{d}\log(y_i).
\end{equation}
\end{itemize}
Then we add 3 cases with more complex interactions where $u_r,v_r\in\mathbb{R}^d$ are fixed deterministic vectors (e.g., sinusoidal patterns scaled), $R=4$  and  $\varepsilon=10^{-2}$. In these cases, the different sums involves the product of positive convex functions by positive concave functions so that the convex-concave structure is preserved. These cases are very difficult.
\begin{itemize}
\item \textbf{High-rank smooth interaction (softplus coupling).}
\begin{equation}\label{eq:case9_nd}
f_6(x,y)=\|x\|_2^2-\|y\|_2^2 + x^\top B y
+\sum_{r=1}^{R}
\mathrm{softplus}(u_r^\top x + b_r)
\Bigl(C_r-\mathrm{softplus}(v_r^\top y + t_r)\Bigr).
\end{equation}
The constants are defined as
\[
C_r = \mathrm{softplus}\!\bigl(\|v_r\|_1 + |t_r|\bigr)+ \varepsilon.
\]

\item \textbf{Sharp nonlinear interaction (exponential coupling).}
\begin{equation}\label{eq:case10_nd}
f_7(x,y)=\|x\|_2^2-\|y\|_2^2 + x^\top B y
+\sum_{r=1}^{R}
\exp(u_r^\top x + b_r)
\Bigl(C_r-\exp(v_r^\top y + t_r)\Bigr),
\end{equation}
where
\[
C_r = \exp\!\bigl(\|v_r\|_1 + |t_r|\bigr)+ \varepsilon.
\]

\item \textbf{Polynomial high-rank interaction (quadratic coupling).}
\begin{equation}\label{eq:case11_nd}
f_8(x,y)=\|x\|_2^2-\|y\|_2^2 + x^\top B y
+\sum_{r=1}^{R}
(u_r^\top x + b_r)^2
\Bigl(C_r-(v_r^\top y + t_r)^2\Bigr),
\end{equation}
where 
\[
C_r = \bigl(\|v_r\|_1 + |t_r|\bigr)^2+ \varepsilon.
\]

\end{itemize}

For Table~\ref{tab:5DICNN}, we keep the same hyperparameters as in the one-dimensional case, except for the penalty scaling factor, which is now set to \(100\) in order to obtain a vanishing penalty term at convergence.
COMONet keeps the full interaction matrix.
The computing time for the saddle network with ICNN is now \(1280\) seconds, while it increases to \(4600\) seconds for P1-ICKAN and \(7200\) seconds for Cubic-ICKAN.
The computing time for COMONet is \(750\) seconds.

\begin{table}[H]
    \centering
\begin{tabular}{r|cc|cc|cc|cc}
\hline
& \multicolumn{2}{c}{Saddle ICNN} &\multicolumn{2}{c}{Saddle P1-ICKAN} &\multicolumn{2}{c}{Saddle Cubic-ICKAN} &\multicolumn{2}{c}{COMONet} \\  \hline
case & mean\_mse & std\_mse &  mean\_mse & std\_mse  & mean\_mse & std\_mse & mean\_mse & std\_mse   \\
\hline  
1 & 7.50e-4 & 9.83e-5 &  1.12e-3 & 9.99e-4 & 3.44e-4 & 2.86e-4  & 5.69e-4 & 2.41e-4\\
2 & 4.32e-4 & 1.12e-4& 9.08e-4 & 7.96e-4  & 3.04e-4 & 2.37e-4 & 3.47e-4 & 1.05e-4 \\
3 & 1.49e-4 & 3.35e-5& 1.83e-3 & 4.27e-3 &  3.95e-4 & 2.31e-4  & 7.55e-4 & 1.61e-3  \\
4 & 6.47e-5 & 3.04e-5&  3.03e-4 & 2.03e-4 &  2.92e-4 & 2.25e-4 &  2.69e-4 & 1.80e-4\\
5 &  6.51e-4 & 2.33e-4 & 1.54e-4 & 8.80e-5 & 5.45e-5 & 2.73e-5 & 2.59e-2 & 4.39e-2\\ 
6& 9.27e-4 & 3.25e-4&6.91e-4 & 1.27e-4&5.00e-4 & 2.45e-4 &  1.23e-3 & 4.43e-4\\ 
7& 3.79e-3 & 5.54e-4 &4.97e-3 & 1.22e-3&2.53e-3 & 9.23e-4 &2.43e-2 & 6.59e-4 \\ 
8& 4.27e-3 & 4.95e-4&8.12e-3 & 1.92e-3& 4.81e-3 & 2.25e-3 & 4.01e-2 & 1.03e-3\\ \hline
    \end{tabular}
    \caption{Convergence of the saddle networks and COMONet in dimension \(5\), using \(N=20\).}
    \label{tab:5DICNN}
\end{table}
The results obtained in dimension \(1\) are broadly confirmed but we observe that on  simpler case 1 and 2, and in the more complex case 6, the COMONet architecture can compete with the saddle architectures.
Results in table \ref{tab:5DICNNGen}, shows that bilinear saddle network allow us to obtain  a modest improvement in accuracy.
\begin{table}[H]
    \centering
\begin{tabular}{rcccccccc}
\hline
case & 1 & 2 & 3 & 4 & 5 &6 &7 & 8  \\ \hline
mean\_mse &8.13e-4&  2.26e-4 & 3.16e-5 & 3.84e-5 &7.03e-4& 9.96e-4 & 3.34e-3 &3.78e-3\\
std\_mse &1.68e-4&  6.68e-5& 7.89e-6& 1.70e-5&2.62e-4&1.17e-4 & 6.43e-4 & 5.87e-4\\
\hline
    \end{tabular}
    \caption{Convergence of the bilinear saddle network in dimension \(5\), using \(N=20\) and ICNN primitives. The results are globally slightly improved.
    \label{tab:5DICNNGen}}
\end{table}

We now study the effect of the parameter \(N\) on the convergence of the saddle network.
We keep only the ICNN architecture for these runs.
Since we are interested in isolating the effect of \(N\) on convergence, preliminary experiments indicate that increasing the number of neurons to \(64\) allows us to obtain slightly more accurate results.
The other hyperparameters are unchanged, except for the number of gradient iterations, which is now set to \(10^6\).

We consider cases \(1\), \(3\),  \(5\) ,\(6\) , \(7\) and \(8\) and plot in Figure~\ref{fig:conVN} the logarithm of the averaged MSE over \(10\) runs as a function of \(\log(N)\) for the saddle network class, with \(N \in \{1,2,4,8,16,32,64,128\}\), together with the associated standard deviation.
As \(N\) increases, the approximation accuracy improves, reaching an average MSE of \(7.68\times 10^{-5}\) for case \(1\), \(2.92\times 10^{-5}\) for case \(3\), $1.9\times 10^{-4}$ for case $5$,  \(8.34 \times 10^{-5}\) for case $6$, $3.53\times 10^{-4}$ for case $7$ and $9.88 \times 10^{-4}$ for case $8$. Interestingly, the standard deviation also decreases overall with \(N\), suggesting that as \(N\) increases, the different runs converge toward very similar functions.
\begin{figure}[H]
    \centering
    \begin{minipage}[b]{0.49\linewidth}
     \includegraphics[width=\linewidth]{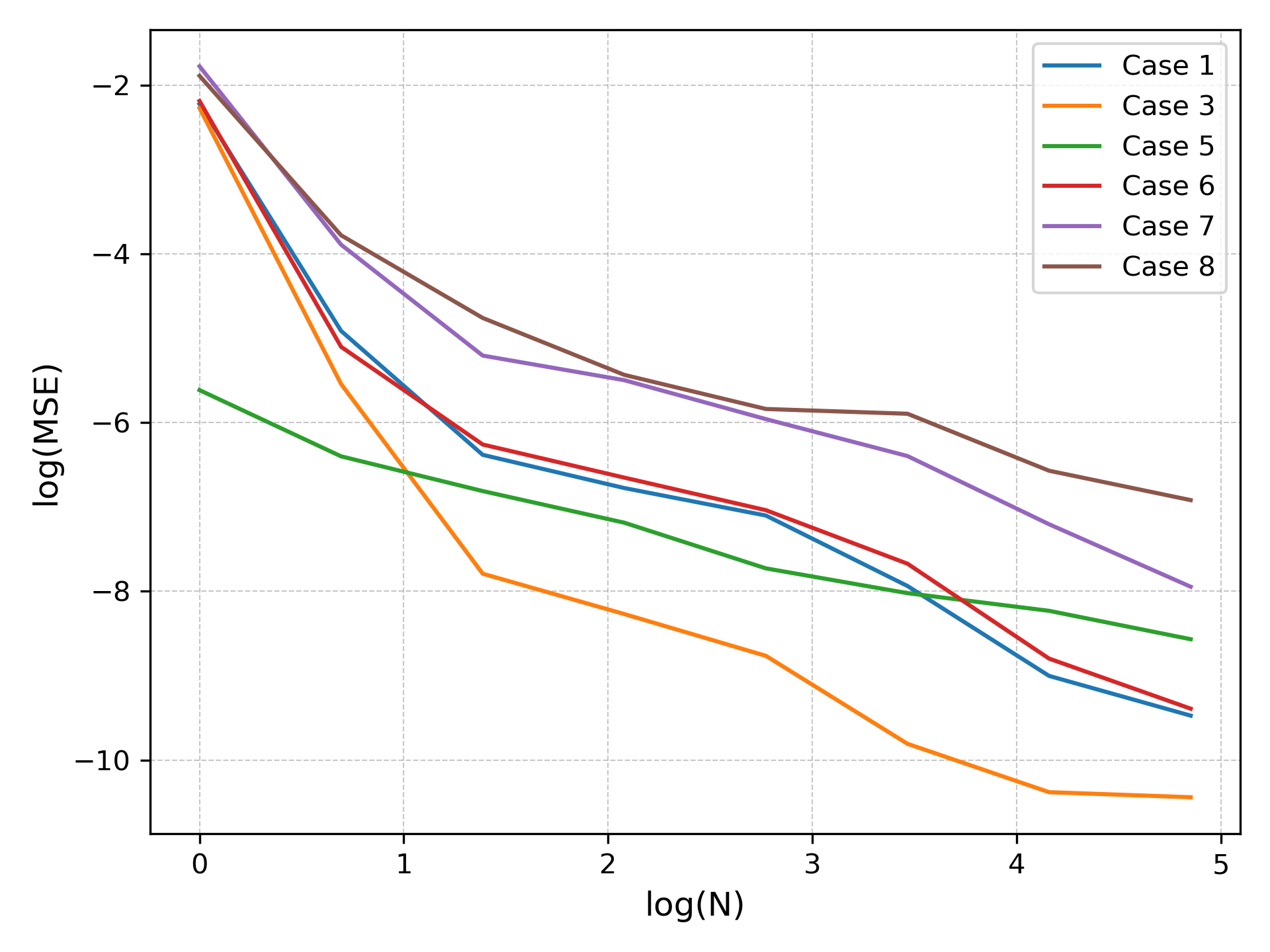}
    \caption*{MSE}   
    \end{minipage}
    \begin{minipage}[b]{0.49\linewidth}
     \includegraphics[width=\linewidth]{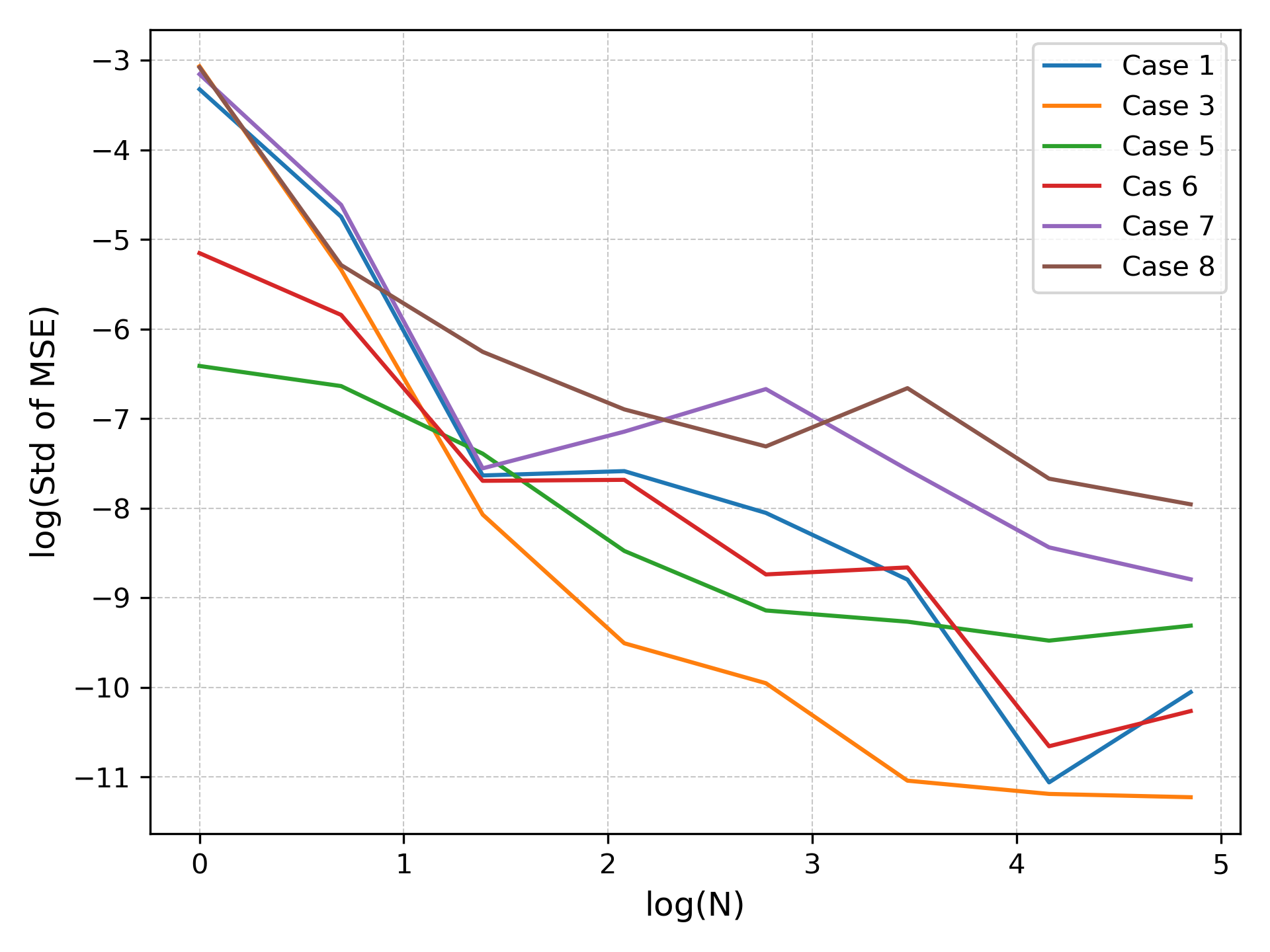}
    \caption*{Std of the MSE}   
    \end{minipage}
    \caption{
    Logarithm of the averaged MSE and of the MSE standard deviation as functions of \(\log(N)\) for the saddle class.}
    \label{fig:conVN}
\end{figure}
Figure \ref{fig:conVNGenNot} indeed shows that adding a bilinear term allow us to use a lower order $N$ for the same accuracy. However it seems that the benefit of the bilinear term appears to vanish as $N$ grows.
\begin{figure}[H]
    \centering
    \begin{minipage}[b]{0.32\linewidth}
     \includegraphics[width=\linewidth]{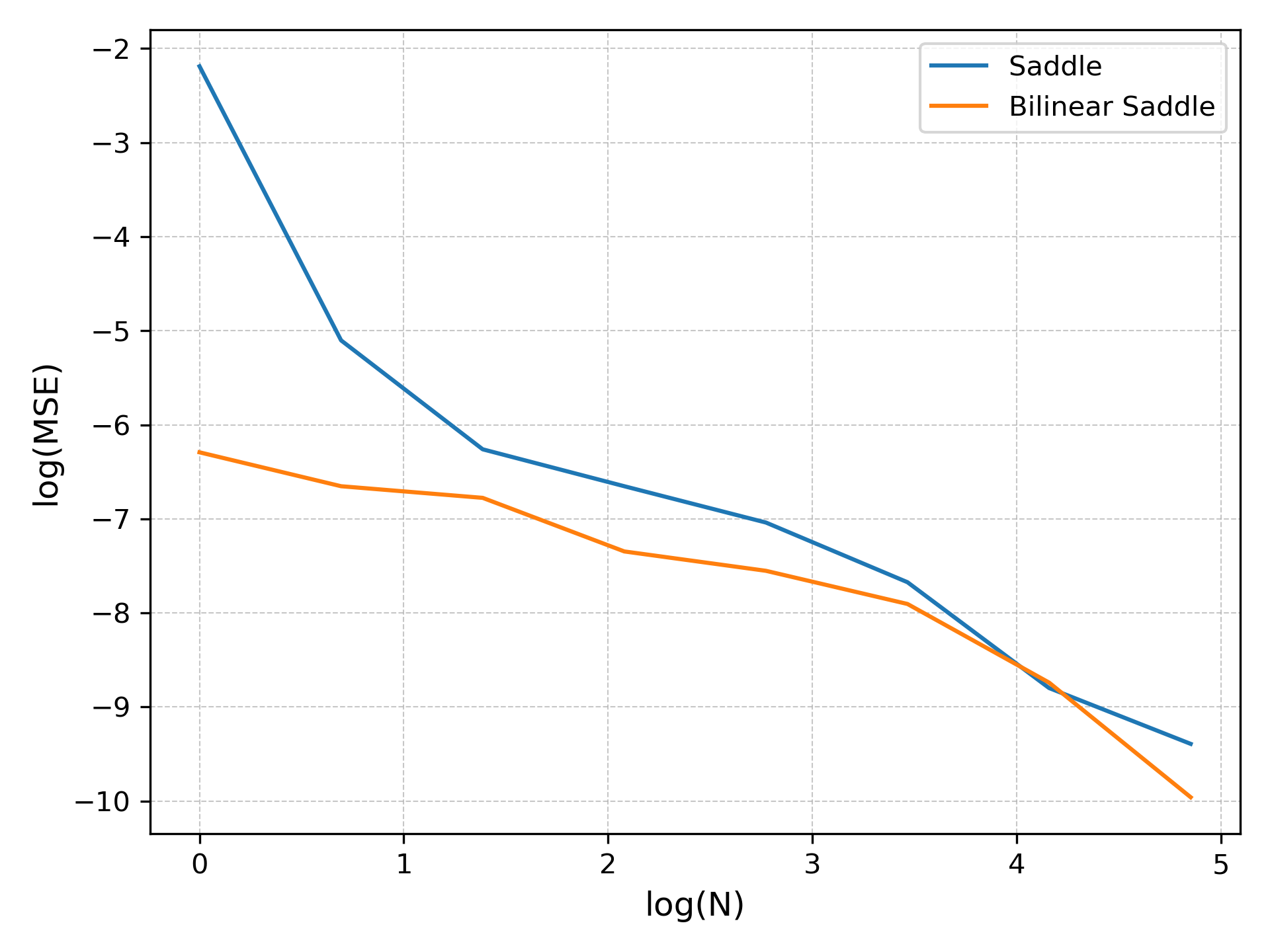}
    \caption*{Case 6}   
    \end{minipage}
    \begin{minipage}[b]{0.32\linewidth}
     \includegraphics[width=\linewidth]{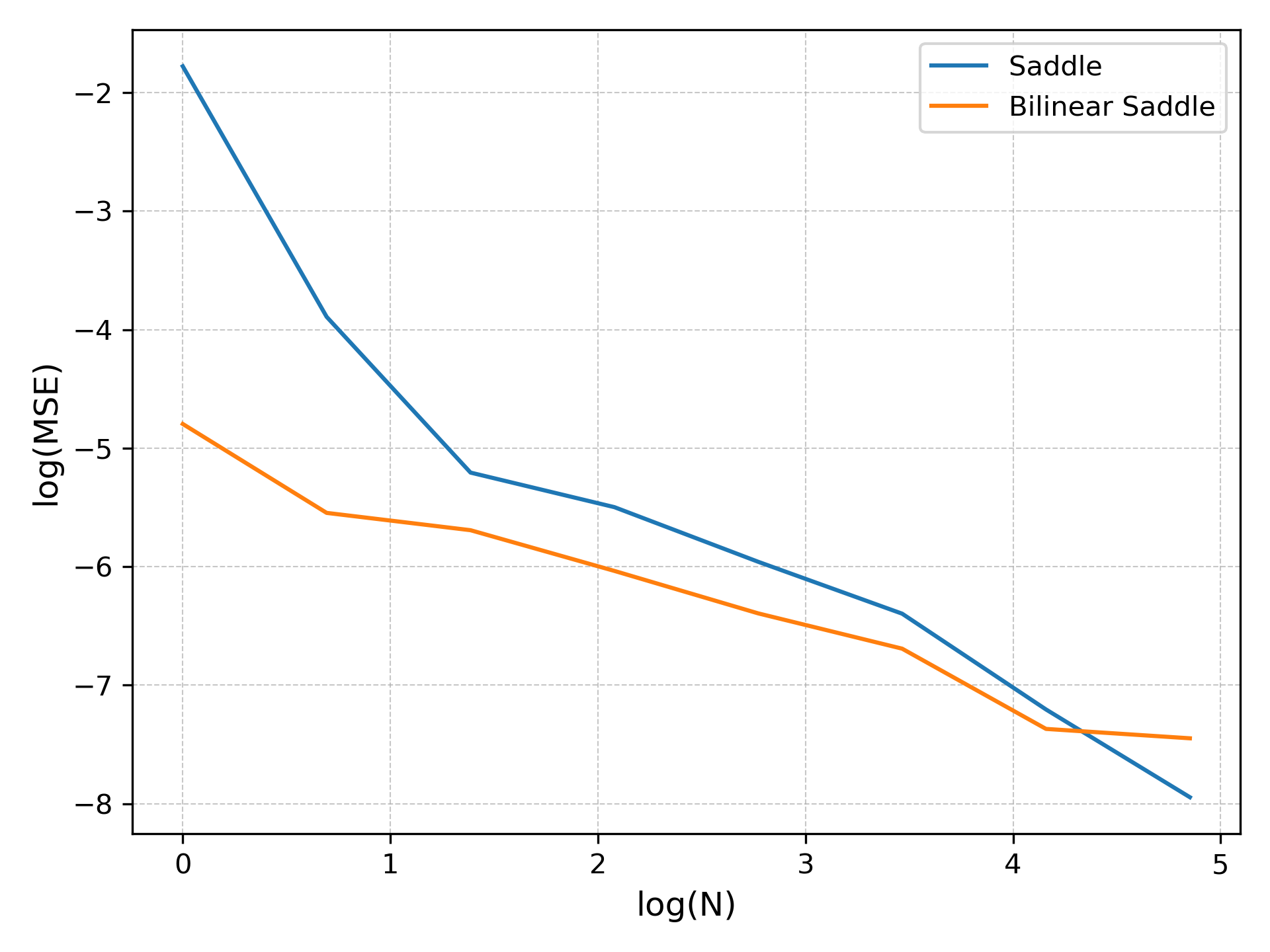}
    \caption*{Case 7}   
    \end{minipage}
   \begin{minipage}[b]{0.32\linewidth}
     \includegraphics[width=\linewidth]{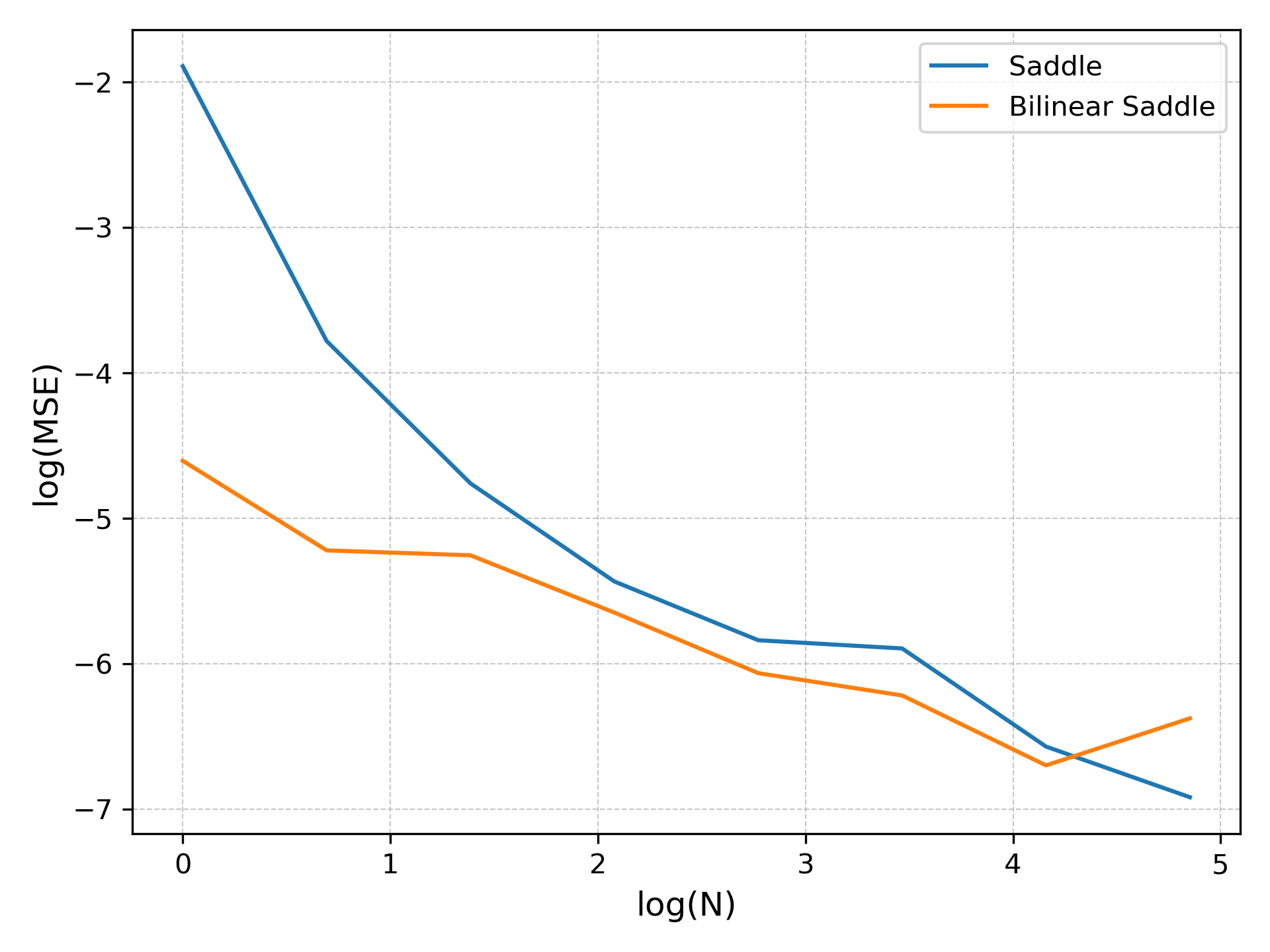}
    \caption*{Case 8}   
    \end{minipage}
    \caption{
    Logarithm of the averaged MSE comparing saddle and bilinear saddle approximation as a function of \(\log(N)\).}
    \label{fig:conVNGenNot}
\end{figure}
\section{Conclusion and perspectives}

The proposed framework suggests that preserving saddle geometry by construction can be achieved with surprisingly simple building blocks, reducing the design of convex-concave models to the choice of a suitable convex primitive and a controlled factorization scheme. This viewpoint shifts the focus from architectural complexity to structural fidelity, which is essential in applications where optimization guarantees are as important as approximation accuracy.

The main open question concerns the intrinsic expressive power of such factorizations in higher dimension. The numerical evidence indicates that increasing the separable rank improves stability and accuracy, but a theoretical characterization of this behavior remains out of reach. Understanding whether a finite-rank representation can capture the full cone of multivariate saddle functions—or quantifying the gap otherwise—would provide a decisive insight into the limits of structure-preserving learning.
Finally, beyond approximation, the relevance of saddle networks lies in their integration into downstream pipelines: robust optimization, adversarial learning, and control all rely on stable min--max formulations. Embedding geometric constraints directly into learned models opens the door to hybrid methods combining data-driven approximation with certified optimization guarantees, a direction that appears particularly promising for high-stakes decision systems.

\bibliographystyle{plain}
\bibliography{biblio}

@article{Goebel2008,
  author = {Goebel, R.},
  title = {Self-dual smoothing of convex and saddle functions},
  journal = {Journal of Convex Analysis},
  volume = {15},
  number = {1},
  pages = {179--190},
  year = {2008}
}

@inproceedings{AmosICNNPMLR,
  author = {Amos, Brandon and Xu, Lei and Kolter, J. Zico},
  title = {Input convex neural networks},
  booktitle = {Proceedings of the 34th International Conference on Machine Learning (ICML)},
  series = {Proceedings of Machine Learning Research},
  volume = {70},
  pages = {146--155},
  year = {2017}
}

@inproceedings{ChenShiZhang2019,
  author = {Chen, Y. and Shi, Y. and Zhang, B.},
  title = {Optimal control via neural networks: A convex approach},
  booktitle = {International Conference on Learning Representations (ICLR)},
  year = {2019}
}

@article{WarinGroupMaxArxiv,
  author = {Warin, Xavier},
  title = {The GroupMax neural network approximation of convex functions},
  journal = {IEEE Transactions on Neural Networks and Learning Systems},
  year = {2023},
  doi = {10.1109/TNNLS.2023.3240183}
}

@article{ICKANArxiv,
  author = {Deschatre, Thomas and Warin, Xavier},
  title = {Input convex Kolmogorov--Arnold networks},
  journal = {arXiv preprint arXiv:2505.21208},
  year = {2025}
}

@article{COMONet2025,
  author = {Kim, H. and Bu, D. and Lee, J.-S.},
  title = {A novel architecture for integrating shape constraints in neural networks (COMONet)},
  journal = {OpenReview},
  year = {2025}
}

@article{SchieleLuxenbergBoydDSP,
  author = {Schiele, Philipp and Luxenberg, Eric and Boyd, Stephen},
  title = {Disciplined saddle programming},
  journal = {arXiv preprint arXiv:2301.13427},
  year = {2023}
}

@article{JuditskyNemirovski2022,
  author = {Juditsky, Anatoli and Nemirovski, Arkadi},
  title = {On well-structured convex--concave saddle point problems and variational inequalities with monotone operators},
  journal = {Optimization Methods and Software},
  volume = {37},
  number = {5},
  pages = {1567--1602},
  year = {2022},
  doi = {10.1080/10556788.2021.1928121}
}

@inproceedings{YangZhangKiyavashHe2020,
  author = {Yang, J. and Zhang, S. and Kiyavash, N. and He, N.},
  title = {A catalyst framework for minimax optimization},
  booktitle = {Advances in Neural Information Processing Systems},
  volume = {33},
  pages = {5667--5678},
  year = {2020}
}

@inproceedings{LinJinJordan2020,
  author = {Lin, Tianyi and Jin, Chi and Jordan, Michael I.},
  title = {Near-optimal algorithms for minimax optimization},
  booktitle = {Proceedings of the 33rd Conference on Learning Theory (COLT)},
  series = {Proceedings of Machine Learning Research},
  volume = {125},
  pages = {2738--2779},
  year = {2020}
}

@article{BotCsetnekSedlmayer2022,
  author = {Bo{\c{t}}, Radu Ioan and Csetnek, Ern{\H{o}} Robert and Sedlmayer, Markus},
  title = {An accelerated minimax algorithm for convex-concave saddle point problems with nonsmooth coupling function},
  journal = {Computational Optimization and Applications},
  volume = {86},
  number = {3},
  pages = {925--966},
  year = {2023},
  doi = {10.1007/s10589-022-00378-8}
}

@article{RiveraWangXu2020,
  author = {Rivera Cardoso, Adrian and Wang, He and Xu, Huan},
  title = {The online saddle point problem and online convex optimization with knapsacks},
  journal = {Mathematics of Operations Research},
  volume = {50},
  number = {1},
  pages = {1--39},
  year = {2025},
  doi = {10.1287/moor.2018.0330}
}

@article{KaweckiSmears2022,
  author = {Kawecki, Ellya L. and Smears, Iain},
  title = {Unified analysis of discontinuous Galerkin and C0-interior penalty finite element methods for Hamilton--Jacobi--Bellman and Isaacs equations},
  journal = {ESAIM: Mathematical Modelling and Numerical Analysis},
  volume = {55},
  number = {2},
  pages = {449--478},
  year = {2021},
  doi = {10.1051/m2an/2020081}
}

@article{EvansSouganidis1984,
  author = {Evans, Lawrence C. and Souganidis, Panagiotis E.},
  title = {Differential games and representation formulas for solutions of Hamilton--Jacobi--Isaacs equations},
  journal = {Indiana University Mathematics Journal},
  volume = {33},
  number = {5},
  pages = {773--797},
  year = {1984}
}

\appendix

\section{Appendix: Proof of the one-dimensional theorem}

We provide a complete and corrected proof of Theorem~\ref{thm:main}. 
The argument relies on a discrete decomposition adapted to convex-concave
structures. The key point is to use a basis of concave ``tent'' functions
rather than simple hinge functions, which avoids spurious mixed terms.

\subsection{Discrete decomposition lemma}
\begin{lemma}[Discrete saddle decomposition in dimension one]
\label{lem:discrete-monge-correct}
Let \(m \ge 1\) and consider a uniform grid with a step $1/m$ giving a grid of points  $(\frac{i}{m},\frac{j}{m})$ for $ i,j \in \{0,\dots,m\}$. 
Let \(A \in \mathbb{R}^{(m+1)\times(m+1)}\) satisfy:
\[
\Delta_x^2 A_{ij} \ge 0,
\qquad
\Delta_y^2 A_{ij} \le 0,
\]
where $\Delta_x ^2 A_{i,j} = A_{i+1,j} + A_{i-1,j}- 2 A_{i,j}$,
$\Delta_y^2 A_{i,j} =A_{i,j+1} + A_{i,j-1}- 2 A_{i,j}$,
for all admissible indices, and assume the mixed condition
\[
\Delta_x^2(-\Delta_y^2 A)_{rs} \ge 0
\quad \text{for all } 1 \le r,s \le m-1.
\]

Then there exist functions:
\begin{itemize}
\item \(E_0, E_m : \{0,\dots,m\} \to \mathbb{R}\), discretely convex,
\item \(B_s : \{0,\dots,m\} \to \mathbb{R}_+\), discretely convex for each \(s=1,\dots,m-1\),
\end{itemize}
such that for all \(i,j\),
\[
A_{ij}
=
E_0(i)L_0(j)
+
E_m(i)L_m(j)
+
\sum_{s=1}^{m-1} B_s(i)\, K_s(j),
\]
where
\[
L_0(j) = 1-\frac{j}{m},
\qquad
L_m(j) = \frac{j}{m},
\qquad
K_s(j) = \min(j,s) - \frac{js}{m}.
\]
\end{lemma}

\begin{proof}

\textbf{Step 1: discrete reconstruction in the \(y\)-direction.}

Fix \(i\) and define \(u_j := A_{ij}\).
Since \(u\) is discretely concave,
\[
\Delta_y^2 u_s \le 0.
\]
Define
\[
c_s := -\Delta_y^2 u_s \ge 0.
\]

We claim that for all \(j\),
\[
u_j
=
\left(1-\frac{j}{m}\right)u_0
+
\frac{j}{m}u_m
+
\sum_{s=1}^{m-1} c_s K_s(j).
\]

Indeed:
\begin{itemize}
\item \(K_s(0)=K_s(m)=0\), so the formula matches the boundary values \(u_0,u_m\);
\item the functions \(L_0,L_m\) are affine in \(j\), hence \(\Delta_y^2 L_0 = \Delta_y^2 L_m = 0\);
\item a direct computation shows
\[
\Delta_y^2 K_s(j)
=
\begin{cases}
-1, & j=s,\\
0, & j\neq s,
\end{cases}
\]
i.e. \(\Delta_y^2 K_s(j) = -\mathbf{1}_{j=s}\).
\end{itemize}

Therefore,
\[
\Delta_y^2 u_j
=
\sum_{s=1}^{m-1} c_s \Delta_y^2 K_s(j)
=
- c_j,
\]
which matches the definition of \(c_j\).
Since the operator \(\Delta_y^2\) with Dirichlet boundary conditions admits a unique inverse on sequences, the representation is uniquely determined by its second differences and boundary values.

Applying this to \(u_j = A_{ij}\) yields
\[
A_{ij}
=
\left(1-\frac{j}{m}\right) A_{i0}
+
\frac{j}{m} A_{im}
+
\sum_{s=1}^{m-1} B_{is} K_s(j),
\]
with
\[
B_{is} := -\Delta_y^2 A_{is} \ge 0.
\]

\medskip
\textbf{Step 2: convexity in the \(x\)-direction.}

By definition,
\[
B_{is} = -\Delta_y^2 A_{is}.
\]
The mixed condition implies
\[
\Delta_x^2 B_{rs}
=
\Delta_x^2(-\Delta_y^2 A)_{rs} \ge 0,
\]
so for each fixed \(s\), the sequence \(i \mapsto B_{is}\) is discretely convex.

Moreover, since \(i \mapsto A_{ij}\) is discretely convex for every \(j\),
the boundary sequences
\[
i \mapsto A_{i0},
\qquad
i \mapsto A_{im}
\]
are discretely convex.

\medskip
\textbf{Step 3: conclusion.}

Define
\[
E_0(i) := A_{i0}, \qquad
E_m(i) := A_{im}, \qquad
B_s(i) := B_{is}.
\]
Then the decomposition follows.
\end{proof}

\subsection{Proof of Theorem~\ref{thm:main}}

\begin{proof}

\textbf{Step 1: reduction.}

By an affine change of variables, we reduce to \(X=Y=[0,1]\).

\medskip
\textbf{Step 2: discretization.}

Let
\[
x_i = \frac{i}{m}, \qquad y_j = \frac{j}{m}, \qquad 0 \le i,j \le m,
\]
and define
\[
A_{ij} := f(x_i,y_j).
\]

Since \(x \mapsto f(x,y)\) is convex and \(y \mapsto f(x,y)\) is concave,
we have
\[
\Delta_x^2 A_{ij} \ge 0, \qquad \Delta_y^2 A_{ij} \le 0,
\]
where $\Delta_x ^2 A_{i,j} = A_{i+1,j} + A_{i-1,j}- 2 A_{i,j}$,
$\Delta_y^2 A_{i,j} =A_{i,j+1} + A_{i,j-1}- 2 A_{i,j}$.

To obtain the mixed discrete condition directly from the distributional
Monge condition, let h=1/m and define the tent functions
\[
\psi_r(x)=(h-|x-x_r|)_+,\qquad
\eta_s(y)=(h-|y-y_s|)_+.
\]
Then in distribution:
\[
\psi_r''=\delta_{x_{r-1}}-2\delta_{x_r}+\delta_{x_{r+1}},
\qquad
\eta_s''=\delta_{y_{s-1}}-2\delta_{y_s}+\delta_{y_{s+1}}.
\]
Hence
\[
\left\langle \partial_{xx}(-\partial_{yy}f),\psi_r\eta_s\right\rangle
=
\left\langle f,-\psi_r''\eta_s''\right\rangle
=
\Delta_x^2(-\Delta_y^2 A)_{rs}.
\]
Since \(\psi_r\eta_s\ge 0\) and
\(\partial_{xx}(-\partial_{yy}f)\ge 0\) in the sense of distributions,
we get
\[
\Delta_x^2(-\Delta_y^2 A)_{rs}\ge 0.
\]
\medskip
\textbf{Step 3: discrete decomposition.}

Applying Lemma~\ref{lem:discrete-monge-correct}, we obtain
\[
A_{ij}
=
E_0(i)L_0(j)
+
E_m(i)L_m(j)
+
\sum_{s=1}^{m-1} B_s(i)\, K_s(j),
\]
where:
\begin{itemize}
\item \(E_0,E_m\) are discretely convex,
\item \(B_s \ge 0\) and discretely convex,
\item \(L_0,L_m,K_s\) are concave.
\end{itemize}

\medskip
\textbf{Step 4: lifting to continuous functions.}

Define \(y_s = s/m\). Introduce the continuous concave functions
\[
a_0(y) = 1-y, \qquad a_m(y) = y,
\]
\[
a_s^{(m)}(y) = m\bigl(\min(y,y_s)-y y_s\bigr).
\]

For each \(k\), define \(e_k\) as the piecewise linear interpolation
of the discrete sequence:
\[
e_0(x_i)=E_0(i), \quad
e_m(x_i)=E_m(i), \quad
e_s(x_i)=B_s(i).
\]

Piecewise linear interpolation preserves convexity, hence
each \(e_k\) is convex on \([0,1]\).

Define
\[
f_m(x,y) := \sum_k e_k(x)\, a_k(y).
\]

The representation given by the lemma reconstructs exactly each discrete
sequence in the $y$-direction, and the interpolation in $x$ preserves
nodal values, this defines a continuous function which interpolates the discrete data
in the sense that \(f_m(x_i,y_j)=A_{ij}\).

\medskip
\textbf{Step 5: convergence.}

With the above scaling, \(a_s^{(m)}(y_j)=K_s(j)\). Hence
\[
f_m(x_i,y_j)=A_{ij}=f(x_i,y_j)
\]
for all grid points.

Moreover, since the functions \(e_k\) are piecewise affine in \(x\) and the
functions \(a_k^{(m)}\) are piecewise affine in \(y\), the function \(f_m\)
coincides on each grid cell with the bilinear interpolant of the values
\[
f(x_i,y_j),\quad f(x_{i+1},y_j),\quad
f(x_i,y_{j+1}),\quad f(x_{i+1},y_{j+1}).
\]

Let \((x,y)\in[x_i,x_{i+1}]\times[y_j,y_{j+1}]\). Then there exist nonnegative
weights \(\lambda_{\alpha\beta}\), \(\alpha,\beta\in\{0,1\}\), summing to one,
such that
\[
f_m(x,y)
=
\sum_{\alpha,\beta\in\{0,1\}}
\lambda_{\alpha\beta}
f(x_{i+\alpha},y_{j+\beta}).
\]
Therefore, by uniform continuity of \(f\),
\[
|f_m(x,y)-f(x,y)|
\le
\omega_f\left(\frac{\sqrt 2}{m}\right).
\]
Hence
\[
\|f_m-f\|_\infty
\le
\omega_f\left(\frac{\sqrt 2}{m}\right)
\to 0.
\]

\medskip
\textbf{Step 6: representation in the saddle class.}

We now show that the approximant \(f_m\) has the admissible saddle-network
structure.

Recall that
\[
f_m(x,y)
=
E_0(x)L_0(y)
+
E_m(x)L_m(y)
+
\sum_{s=1}^{m-1} B_s(x)a_s^{(m)}(y),
\]
where
\[
L_0(y)=1-y,\qquad L_m(y)=y,
\]
and
\[
a_s^{(m)}(y)
=
m\bigl(\min(y,y_s)-yy_s\bigr).
\]

For \(s=1,\dots,m-1\), the functions \(B_s\) are convex and nonnegative,
while \(a_s^{(m)}\) are concave and nonnegative. Hence each product
\[
B_s(x)a_s^{(m)}(y)
\]
is directly of the admissible form
\[
e_i^{cv,+}(x)a_i^{cc,+}(y).
\]

It remains to treat the two boundary terms. Choose constants
\(C_0,C_m\ge 0\) such that
\[
E_0(x)+C_0\ge 0,
\qquad
E_m(x)+C_m\ge 0
\quad\text{for all }x\in[0,1].
\]
Then
\[
E_0(x)L_0(y)
=
\bigl(E_0(x)+C_0\bigr)L_0(y)
-
C_0L_0(y),
\]
and
\[
E_m(x)L_m(y)
=
\bigl(E_m(x)+C_m\bigr)L_m(y)
-
C_mL_m(y).
\]

Since \(E_0+C_0\) and \(E_m+C_m\) are convex and nonnegative, and since
\(L_0,L_m\) are affine nonnegative functions, hence both concave and convex,
the two products
\[
\bigl(E_0(x)+C_0\bigr)L_0(y),
\qquad
\bigl(E_m(x)+C_m\bigr)L_m(y)
\]
are of type
\[
e_i^{cv,+}(x)a_i^{cc,+}(y).
\]

The remaining terms depend only on \(y\):
\[
-C_0L_0(y)-C_mL_m(y).
\]
Since \(L_0\) and \(L_m\) are affine, the function
\[
G(y):=-C_0L_0(y)-C_mL_m(y)
\]
is affine, hence concave on \([0,1]\). Therefore these two terms can be
absorbed into the concave marginal \(G\).

Consequently,
\[
f_m(x,y)
=
\sum_i e_i^{cv,+}(x)a_i^{cc,+}(y)
+
G(y).
\]
Thus \(f_m\) belongs to the saddle class appearing in the statement of the
theorem.

Since \(\|f-f_m\|_\infty\to 0\), choosing \(m\) sufficiently large gives the
desired approximation.

\end{proof}
\end{document}